\documentclass[a4paper,11pt,oneside,reqno]{amsart}
\pdfoutput=1
\usepackage[utf8]{inputenc}
\usepackage[english]{babel}
\usepackage{
   indentfirst
  ,amsfonts
  ,amssymb
  ,amsthm
  ,mathrsfs
  ,leftidx
  ,cite
  ,enumitem
  ,stmaryrd
  ,todonotes
  ,xspace
  ,bbm
  ,epigraph
  ,mathtools
  ,blindtext
  ,wasysym
  ,graphicx
  ,tikz-cd
}
\usepackage[cal=cm, scr=boondoxo]{mathalfa}
\usepackage{hyperref}
\hypersetup{pdfauthor={Stefano Ariotta}, pdftitle={E-infty groups of units and logarithmic derivatives}}
\usepackage[top=4.2cm, bottom=5.6cm, left=4.2cm, right=4.2cm, marginparwidth=80pt]{geometry}
\reversemarginpar
\usepackage{xparse}

\ExplSyntaxOn
\NewDocumentCommand{\makeabbrev}{mmm}
 {
  \yoruk_makeabbrev:nnn { #1 } { #2 } { #3 }
 }

\cs_new_protected:Npn \yoruk_makeabbrev:nnn #1 #2 #3
 {
  \clist_map_inline:nn { #3 }
   {
    \cs_new_protected:cpn { #2 } { #1 { ##1 } }
   }
 }
\ExplSyntaxOff

\makeabbrev{\textsf}{sf#1}{a,b,c,d,e,f,g,h,i,j,k,l,m,n,o,p,q,r,s,t,u,v,w,x,y,z,A,B,C,D,E,F,G,H,I,J,K,L,M,N,O,P,Q,R,S,T,U,V,W,X,Y,Z}

\makeabbrev{\mathfrak}{fk#1}{a,b,c,d,e,f,g,h,i,j,k,l,m,n,o,p,q,r,s,t,u,v,w,x,y,z,A,B,C,D,E,F,G,H,I,J,K,L,M,N,O,P,Q,R,S,T,U,V,W,X,Y,Z}
\makeabbrev{\mathcal}{cal#1}{A,B,C,D,E,F,G,H,I,J,K,L,M,N,O,P,Q,R,S,T,U,V,W,X,Y,Z}
\makeabbrev{\mathbb}{bb#1}{A,B,C,D,E,F,G,H,I,J,K,L,M,N,O,P,Q,R,S,T,U,V,W,X,Y,Z}
\makeabbrev{\mathscr}{scr#1}{A,B,C,D,E,F,G,H,I,J,K,L,M,N,O,P,Q,R,S,T,U,V,W,X,Y,Z}

\newcommand{\xto}[1]{\xrightarrow{#1}}

\newcommand{\To}{\Rightarrow}

\renewcommand{\phi}{\varphi}

\DeclareMathOperator{\Fun}{Fun}

\DeclareMathOperator{\GL}{GL_1}
\DeclareMathOperator{\gl}{\mathit{gl}_1}

\DeclareMathOperator{\Map}{Map}
\DeclareMathOperator{\ho}{ho}
\DeclareMathOperator{\Binf}{B^{\infty}}
\DeclareMathOperator{\coker}{coker}
\DeclareMathOperator{\nerve}{N}
\DeclareMathOperator{\Symm}{Sym^{*}}
\DeclareMathOperator{\sym}{Sym}
\DeclareMathOperator{\free}{\mathcal{F}}
\DeclareMathOperator{\bpshift}{\textsc{Sh}}

\def\tensor{\qopname\relax m{\otimes}}

\newcommand{\gp}{^{\mathrm{gp}}}
\newcommand{\calg}{\mathrm{CAlg}}

\newcommand{\Sp}{\mathrm{Sp}}

\newcommand{\Mod}{\mathrm{Mod}}

\newcommand{\inv}{\mathrm{inv}}
\newcommand{\id}{\mathrm{id}}

\newcommand{\ominfty}{\Omega ^{\infty}}
\newcommand{\susinftyp}{\Sigma ^{\infty}_+}
\newcommand{\susinfty}{\Sigma ^{\infty}}
\newcommand{\grp}{\mathrm{CGrp}}
\newcommand{\mon}{\mathrm{CMon}}
\newcommand{\einf}{\mathbb{E}_{\infty}}
\newcommand{\pt}{\mathrm{pt}}

\newcommand{\cat}{$\infty$-cat\-e\-go\-ry\xspace}
\newcommand{\cats}{$\infty$-cat\-e\-gories\xspace}

\newcommand{\wt}[1]{\widetilde{#1}}
\newcommand{\spaces}{\calS}
\newcommand{\operad}{^{\otimes}}

\newcommand{\pre}[1]{\leftidx{^{+}}{#1}}

\newcommand{\aug}{_\mathrm{aug}}
\newcommand{\units}{^\mathrm{u}}
\newcommand{\pr}{\mathrm{pr}}

\newtheoremstyle{reference}
   {}
   {}
   {}
   {}
   {\bfseries}
   {}
   {.1em}
   {\thmname{#1}
    \thmnumber{#2}
    \thmnote{{\sc [#3]}}}

\theoremstyle{reference}
  \newtheorem{thm}{Theorem}[section]
  \newtheorem{prop}[thm]{Proposition}
  \newtheorem{lemma}[thm]{Lemma}
  \newtheorem{cor}[thm]{Corollary}
  \newtheorem{defi}[thm]{Definition}
  \newtheorem{rem}[thm]{Remark}
  \newtheorem{ex}[thm]{Example}
  
  \newtheorem{notat}[thm]{Notation}
  
  \newtheorem{constr}[thm]{Construction}
  
  \newtheorem*{thm*}{Theorem}
  \newtheorem*{cor*}{Corollary}
  \newtheorem*{counterex*}{Counterexample}
  \newtheorem*{defi*}{Definition}
  \newtheorem*{ex*}{Example}
  \newtheorem*{exercise*}{Exercise}
  \newtheorem*{lemma*}{Lemma}
  \newtheorem*{notat*}{Notation}
  \newtheorem*{prop*}{Proposition}
  \newtheorem*{question*}{Question}
  \newtheorem*{rem*}{Remark}
  \newtheorem*{them*}{Theorem}
  \newtheorem*{constr*}{Construction}
  \newtheorem*{guess*}{Guess}

\hypersetup{
    colorlinks=true,
    linkcolor=teal,
    citecolor=teal,
    urlcolor=teal
}

\allowdisplaybreaks

\author{Stefano Ariotta}
\title[$\einf$-groups of units and logarithmic derivatives]
  {Deformations of $\einf$-groups of units and logarithmic derivatives
  of $\einf$-rings}
\date{}

\begin{document}
\maketitle
\begin{abstract}
  We extend a classical fact about deformations of groups of units of
  commutative rings to $\einf$-ring spectra, and we use this
  result to provide a map of spectra generalizing the ordinary logarithmic
  derivative induced by an $R$-module derivation.
\end{abstract}
\setcounter{tocdepth}{2}
\tableofcontents
\section*{Introduction}
\addtocontents{toc}{\protect\setcounter{tocdepth}{0}}
Stable homotopy theory provides, among other things, a framework to develop a
homotopy coherent version of ordinary commutative algebra.
Instances of this generalization are by now abundant in the literature.
The aim of this paper is to extend to this context the notion of
logarithmic derivatives, by means of a generalization of the following classical
result (see Proposition \ref{classicalresult} for a more detailed statement,
and a proof).

\begin{prop*}
  Given a square-zero extension of commutative rings $\widetilde R \to R$ with
  kernel $I$, there exists an induced short exact sequence of Abelian groups
  \begin{displaymath}
      0 \longrightarrow I \longrightarrow \GL \widetilde R
      \longrightarrow \GL R \longrightarrow 0.
  \end{displaymath}
\end{prop*}

The main motivation for the definition of homotopy coherent logarithmic
derivatives is the investigation of lifts of Adams
operations on (suitable localizations of) K-theory
to morphisms of spectra, developing on current work in progress by
Barwick-Glasman-Mathew-Nikolaus.

To work with homotopy theoretic tools, we will use the language of \cats,
as developed in \cite{HTT}, and we will set our discussion of stable homotopy
theoretic ideas in the framework developed, using this language, in \cite{HA}.

To guide the intuition in the homotopy coherent setting, it is useful to
think of the \cat of spaces as playing the role pertaining to the
category of sets in the ordinary context. We can express the analogy
with ordinary commutative algebra by thinking of the \cat of spectra as the
analogue of the Abelian category of Abelian groups and the derived category
of the integers at once (and more generally
of stable \cats as the analogue of both Abelian categories and derived
categories). Along these lines,
$\einf$-rings become the counterparts of ordinary commutative rings.

With these ideas in mind, our goal will be to prove the following
generalization of the above Proposition
(see Theorem \ref{big-thm}; the reasons for the connectivity conditions will
be clearified in Remark \ref{connectivity-hp}).

\begin{thm*}
  Let $R$ be a connective $\einf$-ring, and let $\wt R \to R$ be a square-zero
  extension by a connective $R$-module $M$.
  Then, there exists a co/fiber sequence
  $$M \to \gl\wt R \to \gl R$$
  in the \cat $\Sp$.
\end{thm*}

By virtue of the above theorem, we will be able to define a map of
spectra $\log_\partial\colon\gl R \to M$ from any connective $\einf$-ring
$R$ and any derivation (as defined in \cite[7.4.1]{HA} and reviewed in Section
\ref{subs-sq-zero}) $\partial$ of $R$ into an $R$-module $M$.

In Section \ref{bigproof}, we set up what we need in order to state the
theorem, then we give two alternative proofs of it, one recovering the ordinary
proposition as a particular case, the other leveraging the homotopy coherent
statement from the ordinary one.
In Section \ref{applications}, we
define our homotopy coherent logarithmic derivative, and show how it
generalizes the ordinary one.
Appendix \ref{appendix-classical}, contains a proof of the ordinary
result.
Finally, in Appendix \ref{prelim}, we recall a few
results about symmetric monoidal \cats and modules over commutative algebra
objects.

\subsection*{Prerequisites}
We assume the reader is familiar with the language of \cats, as developed
in \cite{HTT}. Notably, we make free use of the Adjoint Functor Theorem
(see \cite[5.5.2.9]{HTT}).
Even if we will recall some definitions, we will assume that the reader is
familiar with the theory of stable \cats and of $\infty$-operads as developed
in \cite{HA}, and in particular with the \cat of spectra and its symmetric
monoidal structure given by the smash product.
Some of the notations used in this paper differ from the ones in \cite{HA}.
However, we use the same notations and terminology of \cite{HA} for all the
concepts we do not explicitly recall or introduce here.

\subsection*{Relation to previous work}
A result essentially equivalent to Theorem \ref{big-thm} appeared,
using rigid models to deal with homotopy coherent structures, as
\cite[Lemma 11.2]{rognes2009topological},
and the proof sketched in \cite[Remark 11.3]{rognes2009topological} is strictly
related to the proof we present in Section \ref{subs-top-down}.

\subsection*{Aknowledgements}
Most of the results of this paper were obtained as part of the author's
M.Sc.\ Thesis at the University of Bonn in 2018. I am extremely grateful to my
advisor, Thomas Nikolaus,
for proposing me this topic and for his thorough support and his invaluable
advice. I would like to thank Jack Davies, Simone Fabbrizzi,
Giulio Lo Monaco, Fosco Loregian, Riccardo Pedrotti and Mauro Porta
for the insightful discussions I had the pleasure to be involved in while
writing this article.
\addtocontents{toc}{\protect\setcounter{tocdepth}{2}}

\section{Groups of units and square-zero extensions}\label{bigproof}
We assume the reader is familiar with the theories of stable \cats and of
$\infty$-operads, as developed in \cite{HA}. For the convenience of the
reader, and to fix notations, we recall some definitions and some results
about symmetric monoidal \cats and modules over (homotopy coherent)
commutative algebra objects in Appendix \ref{prelim}, and we refer to \cite{HA}
for all the concepts there undefined.

We focus our attention to the presentable symmetric monoidal
\cats of spaces and spectra. In Section \ref{definitionz}, we begin by
defining the
objects we want to investigate, namely groups of units and square-zero
extensions of $\einf$-rings, and by establishing some of their basic
properties.
Then, in Section \ref{subs-setup}, we begin our proof of Theorem \ref{big-thm}.
In Section \ref{subs-top-down} and \ref{subs-bottom-up} we provide two
alternative proofs, one at the spectral level, and one at the space level.

\subsection{Definitions and basic facts}\label{definitionz}
Most of the material presented in this section, is an adaptation of definitions
and results presented using rigid models for the categories of spectra
in \cite{abghr14} and \cite{rez06}. In particular, our treatment of groups
of units and group $\einf$-rings is inspired by the presentation given in
\cite{abghr14}, although our proofs are independent of the rigid model
results, and are entirely developed in the $\infty$-categorical setting.
Also, we recollect various facts regarding the properties of groups of units
of $\einf$-rings following closely \cite{rez06}, working out explicitly some
details left to the reader therein. Finally, we recall some definitions and
basic facts about square-zero extensions of $\einf$-rings from \cite{HA}.

We will reserve the notations
$\mon$, $\grp$ and $\calg$ for the \cats $\mon(\spaces)$, $\grp(\spaces)$
and $\calg(\Sp)$, respectively (see Remark \ref{prototypical-ex}).

\begin{prop}\label{umonoid-ringmonoid}
  The symmetric monoidal adjunction $\susinftyp \dashv \ominfty$, restricts to
  an adjunction
  $$\bbS[-] : \mon \rightleftarrows \calg : \ominfty_m$$
  such that the diagram
  $$
  \begin{tikzcd}
  \mon \ar[r, "\text{$\bbS[-]$}"] \ar[d] & \calg \ar[d] \\
  \spaces \ar[r, "\susinftyp"] & \Sp
  \end{tikzcd}
  $$
  (where the vertical arrows are the forgetful functors) commutes.
\end{prop}
\begin{proof}
  Recall that
  $\mon\simeq\mathrm{Alg}_{\einf\operad}(\spaces)\subset\Fun_{\einf\operad}
  \left(\einf\operad,\spaces^\times\right)$.
  The functor $\bbS[-]$ is obtained by postcomposition with the symmetric
  monoidal functor $\spaces^{\times} \to \Sp\operad$ induced by $\susinftyp$,
  and $\ominfty_m$ is obtained analogously.

  As explained in \cite{RV1,RV4}, to give an adjunciton between \cats is
  equivalent to give an adjunction in the underlying 2-category of
  quasi-categories.
  Given unit and counit transformations for the adjunction $\susinftyp \dashv
  \ominfty$, we obtain unit and counit transformations for $\bbS[-]$ and
  $\ominfty_m$ by precomposition with (the underlying 1-cells of) sections of
  the symmetric monoidal structure maps, thus satisfying the triangle
  identities automatically.
\end{proof}

\begin{defi}
  Given an $\einf$-ring $R$, we call $\ominfty_m R$ its \emph{underlying
  multiplicative $\einf$-space}.
  Given an $\einf$-space $X$, we call $\bbS[X]$ its \emph{monoid $\einf$-ring}.
\end{defi}

A few words on notation and terminology are overdue. To see the reasons for
the analogy with the ordinary case, consider an $\einf$-ring $R$.
By definition, $R$ is a commutative algebra object in the \cat of spectra.
The functor $\ominfty_m$ acts by ``forgetting'' the additive
structure of $R$ (falling back from spectra
to spaces), but still remembering its multiplicative structure,
as it takes values in the \cat of commutative algebra objects in
$\spaces$.
The analogy is particularly clear if we restrict to connective $\einf$-rings.
By Remark \ref{infinite-loops}, the \cats $\Sp^{cn}$ and $\grp$ are
equivalent, hence the \cat of connective $\einf$-rings $\calg^{cn}$ is
equivalent to the \cat $\calg(\grp)$.
As the restriction of $\ominfty$ to $\Sp^{cn}$ corresponds to the forgetful
functor $\grp \to \spaces$ under this equivalence, the functor
$\ominfty_m$ restricts to $\calg(\grp) \to \calg(\spaces) \simeq \mon$.

\begin{prop}\label{maximal-subgrp}
  The forgetful functor $\grp \to \mon$ admits a right adjoint, denoted
  $$(-)^\times\colon \mon \to \grp.$$
\end{prop}
\begin{proof}
  This is a direct consequence of \cite[5.2.6.9]{HA} and the Adjoint Functor
  Theorem.
\end{proof}

\begin{rem}\label{explicit-subgrp}
  The functor $(-)^\times$ can be explicitly described as the functor sending
  an $\einf$-space $M$ to its connected components which are invertible in
  the commutative monoid $\pi_0 M$ (see Remark \ref{ordinary-monoids}), and
  the counit of the adjunction as the direct summands inclusion. To see
  this, let us momentarily denote by $M\units$ those connected components of
  $M$ which are invertible in $\pi_0 M$, and let us denote by $\iota\colon
  M\units \to M$ the direct summands inclusion. Clearly, given any
  $\einf$-group $G$, postcomposition with $\iota$ determines a map
  $$\Map_{\grp}(G,M\units) \xto{\iota_*} \Map_{\mon}(G,M).$$
  On the other hand, any element in $\Map_{\mon}(G,M)$ induces a morphism
  of ordinary monoids $\pi_0 G \to \pi_0 M$, which, since $\pi_0 G$ is a group,
  factors through $(\pi_0 M)^\times$. As a consequence, the actual morphism
  $G\to M$ we started from has to factor through $\iota$. A straightforward
  check shows that this factorization gives a homotopy inverse to $\iota_*$.
\end{rem}

Using the notation given in Appendix \ref{prelim},
we think of $(-)\gp$ as group completion and of $(-)^\times$ as the maximal
subgroup of an $\einf$-space.

\begin{notat}\label{notat-gl}
  We denote the composite right adjoint $(-)^\times \circ \ominfty_m$ by
  $$\GL\colon \calg \to \grp$$
  and its left adjoint, using a notation analogous to the common one for
  the ordinary case, again by $\bbS[-]$; we denote the composite functor
  $\Binf \GL$ by
  $$\gl\colon \calg \to \Sp$$
  (see Proposition \ref{ladjs-factorization} for the definition of $\Binf$).
  Given an $\einf$-ring $R$, we will use the term \emph{$\einf$-group of
  units} of $R$ to refer to $\GL R$; we will call $\gl R$ the \emph{units-group
  spectrum} of $R$. We will sometimes just use the term \emph{group of units}
  to refer to either of the two.
\end{notat}

\begin{notat}
  Let $X$ be a space, and let $R$ be an $\einf$-ring. We will use the notation
  $$H^\bullet \left ( X ; \gl R \right )$$
  to denote the cohomology groups
  $$ \pi_0 \Map_\Sp \left ( \susinftyp X, \gl R [\bullet] \right ).$$
  Given any pointed space $X$, we will denote by
  $$\wt H^\bullet \left ( X ; \gl R \right )$$
  the reduced cohomology groups
  $$ \pi_0 \Map_\Sp \left ( \susinfty X, \gl R [\bullet] \right ).$$
\end{notat}

\begin{rem}\label{GL-to-mult}
  The counit of the adjunction given by Remark \ref{explicit-subgrp} gives
  a natural transformation $\iota\colon \GL \To \ominfty_m$. It follows from the
  same remark that, given any $\einf$-ring $R$, the map $\iota_R$ fits into
  a Cartesian square
  $$
  \begin{tikzcd}
  \GL R \ar[r, "\iota_R"] \ar[d] & \ominfty_m R \ar[d] \\
  (\pi_0 R)^\times \ar[r] & \pi_0 R
  \end{tikzcd}
  $$
  in the \cat $\mon$ (where both $(\pi_0 R)^\times$ and $\pi_0 R$ are
  considered as discrete $\einf$-spaces) and thus in the \cat $\spaces$.
  In particular, as all the path components of a grouplike H-space
  are homotopy equivalent, $\gl R$ is a connective spectrum having as
  homotopy groups:
  \begin{displaymath}
  \pi _n (\gl R) =
  \begin{cases}
      \pi _0 (R)^\times \quad &\text{for } n=0 \text{;}\\
      \pi _n (R) &\text{for } n \geq 1\text{.}
  \end{cases}
  \end{displaymath}
\end{rem}

\begin{rem}
  As a consequence of Remark \ref{GL-to-mult}, we have that for any
  $\einf$-ring $R$, the following isomorphism of groups holds
  $$H^0 ( X ; \gl R) \simeq \left ( R^0 X \right )^\times.$$
  In fact, as the (multiplicative) commutative monoid structure of
  $$R^0 X \simeq \pi_0 \Map_\spaces (X,\ominfty R)$$ is determined by the
  commutative monoid object structure of $\ominfty R$ in the homotopy
  category $\ho\spaces$, the
  maps $X \to \ominfty R$ factoring through $\GL R$ correspond exactly to
  the invertible elements in $R^0 X$.
\end{rem}

\begin{rem}\label{natural-augmentation}
  Given any $\einf$-monoid $X$, the $\einf$-ring $\bbS[X]$ comes with an
  augmentation map $\bbS[X]\to\bbS$ that is natural and compatible with all
  the relevant adjunctions (see Remark \ref{explicit-augmentation} for an
  explicit description of this map).

  More precisely, both functors denoted $\bbS[-]$ factor through
  the \cat $\calg_{/\bbS}$ of
  \emph{augmented $\einf$-rings}. In fact, since $\bbS[0] \simeq \bbS$,
  we have an induced functor
  $$\grp \simeq \grp_{/0} \to \calg_{/\bbS}$$
  whose composition with the forgetful functor $\calg_{/\bbS} \to \calg$ is
  precisely $\bbS[-]$.
  By \cite[1.2.13.8]{HTT}, the forgetful functor $\calg_{/\bbS} \to \calg$
  preserves colimits; hence, by the Adjoint Functor Theorem it
  admits a right adjoint. As a consequence, we have that the adjunction
  $\bbS[-] \dashv \ominfty_m$ (resp. $\bbS[-] \dashv \GL$) factors as a
  composite adjunction
  $$\mon\rightleftarrows\calg_{/\bbS}\rightleftarrows\calg$$
  (resp.
  $\grp\rightleftarrows\calg_{/\bbS}\rightleftarrows\calg$).
\end{rem}

\begin{rem}\label{pointed-splitting}
  Given any connected pointed space $X$, the terminal morphism $X\to\pt$ admits
  as a section the basepoint inclusion $\pt \to X$.
  By applying the free basepoint functor $(-)_+$ introduced in Proposition
  \ref{ladjs-factorization} to the basepoint inclusion, we obtain a pointed
  map $S^0 \to X_+$, which in turn fits into a cofiber sequence of pointed
  spaces
  $$S^0 \to X_+ \to X$$
  where the last map is the counit of the free-forgetful adjunction
  $\spaces \rightleftarrows \spaces_*$.
  By applying the functor $\susinfty$ to the above sequence, we get a
  co/fiber sequence in $\Sp$:
  $$\bbS \to \susinftyp X \to \susinfty X.$$
  Since the first map admits a retraction (the image under $\susinftyp$ of
  $X\to\pt$), the co/fiber splits and
  $$\susinftyp X \simeq \susinfty X \oplus \bbS.$$
  Following \cite{rez06}, we will denote by
  $$\gamma _X\colon \susinfty X \to \susinftyp X$$ the section (well-defined
  up to homotopy)
  of the projection $\susinftyp X \to \susinfty X$.
\end{rem}

\begin{rem}\label{explicit-augmentation}
  The above remark lets us give an explicit description of the
  augmentation map of a monoid or group $\einf$-ring. Let $X$ be an
  $\einf$-space, and let $a\colon \bbS[X]\to\bbS$ be its monoid $\einf$-ring,
  together with the augmentation map given by Remark
  \ref{natural-augmentation}. Looking at the map underlying $a$ in $\Sp$, we
  have that it is obtained by applying the functor
  $\susinftyp$ to the terminal morphism $X \to \pt$, hence it is given by the
  morphism
  $$\bbS \oplus \susinfty X \to \bbS\oplus 0
    \simeq \bbS$$
  acting as the identity on $\bbS$, and as the zero morphism elsewhere,
  informally written as $\begin{pmatrix}\id_\bbS &0\\0&0\end{pmatrix}$.
\end{rem}

\begin{notat}
  Let $A$ and $B$ be spaces, let $X$ and $Y$ be pointed spaces and let $H$ and
  $K$ be spectra.
  We will sometimes use the notations
  $$[A,B]\text{,} \quad [X,Y]_* \quad \text{and} \quad [H,K]$$
  to denote
  $$\pi_0 \Map_{\spaces}(A,B)\text{,} \quad
  \pi_0 \Map_{\spaces_*}(X,Y) \quad \text{and} \quad \pi_0 \Map_\Sp(H,K)$$
  respectively.
\end{notat}

\begin{rem}\label{rem-splitting-subgroup}
  Let $E$ be a spectrum. The splitting discussed in Remark
  \ref{pointed-splitting} induces the usual direct sum decomposition
  for the unreduced cohomology of a pointed space, i.e.\ for every $n \in \bbZ$,
  it induces an isomorphism
  \begin{displaymath}
  \begin{split}
    E^n(X) &\simeq \Big[\susinftyp X[-n] ,E \Big] \\
    &\simeq \Big[\susinfty X[-n], E \Big] \oplus
    \Big[\bbS[-n], E \Big]\\
    &\simeq \vphantom{\Big[} \wt E^n(X) \oplus \pi_{-n} E.
  \end{split}
  \end{displaymath}
  In particular, this, together with the isomorphism $H^0 (X; \gl R) \simeq
  \left( R^0 X \right)^{\times}$ implies that
  \begin{displaymath}
  \begin{split}
  H^0 (X;\gl R) &\simeq
  \left( \widetilde R^0 X \oplus \pi_0 R \right)^{\times}\\
  &\simeq \left[\susinfty X, \gl R\right]\oplus\left[\bbS,\gl R\right]\\
  &\simeq \widetilde
  H^0(X; \gl R) \oplus \left( \pi_0 R \right)^{\times}.
  \end{split}
  \end{displaymath}
  Thence, the retraction $\susinftyp X\to\susinfty X$ induces an
  isomorphism
  $$\widetilde H^0 \left(X;\gl R\right) \simeq
  \left( \widetilde R^0 X + 1_{\pi_0 R}\right)^{\times}
  \subset\left( \widetilde R^0 X \oplus \pi_0 R \right)^{\times}.$$
\end{rem}

\begin{defi}\label{def-bpshift}
  Let $(X,*)$ be a pointed space and let $E$ be a spectrum. Given an
  unpointed map
  $f\colon X\to\ominfty E$, we define the \emph{basepoint shift} $\bpshift(f)$
  of $f$ to be the composite pointed map
  \begin{displaymath}
  X \simeq X\times\pt\xto{\id\times *} X\times X
    \xto{f\times\left(\inv\circ f\right)}
  \ominfty E \oplus \ominfty E \xto{\mu} \ominfty E
  \end{displaymath}
  where $\mu$ denotes the multiplication map on $\ominfty E$, and
  $\inv$ denotes its inversion map (both given equivalently by the fact that
  $\ominfty E$ is the space underlying the additive $\einf$-group
  $\ominfty_a E$\footnote{see Notation \ref{notat-adj-appendix}}, or
  by Remark \ref{additive-spectra}).
  Informally, we have that $$\bpshift(f)(x) = f(x) - f(*).$$
\end{defi}

\begin{rem}\label{sh-group-hom}
  The assignment $\bpshift\colon f \mapsto \bpshift(f)$ determines a function
  $$
  \begin{tikzcd}
  {[X,\ominfty E]} \ar[r, "\bpshift"] \ar[d, "\simeq"]&
    {[X,\ominfty E]_*} \ar[d, "\simeq"] \\
  \wt E^0 X \oplus \pi_0 E \ar[r] & \wt E^0 X
  \end{tikzcd}
  $$
  corresponding to the projection of unreduced cohomology to the reduced
  cohomology direct summand.
  That is, the map $$E^0 X \to \wt E^0 X$$ induced by
  $\bpshift$ corresponds to the map
  $$[\susinftyp X, E] \xto{(\gamma_X)_*} [\susinfty X, E]$$
  induced by precomposition with the inclusion $\gamma_X$ introduced in Remark
  \ref{pointed-splitting}. In other words, given any map
  $f \in [X,\ominfty E]$, the map $\bpshift(f)$ is the adjoint of the composite
  map $f^\flat \circ \gamma_X \in [\susinfty X, E]$, $f^\flat$ being the
  adjoint map of $f$. In particular, $\bpshift$ preserves whatever additive
  or multiplicative structure the spectrum $E$ may induce on the set
  $[\susinftyp X, E]$.
\end{rem}

Let $R$ be an $\einf$-ring.
The map $$\iota_R\colon \GL R \to \ominfty_m R$$
given by Remark \ref{GL-to-mult} is of course unpointed.
Given any pointed space $X$, we have an
induced morphism $$(\iota_R)_*\colon [X,\GL R]_*\to[X,\ominfty R]$$ and,
by postcomposition
with $\bpshift$, we get a morphism $$\bpshift\circ(\iota_R)_*\colon [X,\GL R]_*
\to [X,\ominfty R]_*.$$
We want to prove that for $X=S^k$ and $k\geq 1$, the above map is an
isomorphism of Abelian groups.

\begin{prop}
\label{pigl}
  Let $R$ be an $\einf$-ring.
  Given any $k \geq 1$, the map $\textsc{Sh}\circ (\iota_R)_*$ induces an
  isomorphism
  $$[S^k,\GL R]_* \simeq [S^k,\ominfty R]_*.$$
\end{prop}

\begin{proof}
  $[S^k,\GL R]_*$ consists of homotopy classes of maps $S^k \to \GL R$ sending
  $S^k$ to the path component $1_{\pi_0 R} \in \left(\pi_0 R \right)^{\times}$.
  By postcomposing with $\iota_R\colon \GL R \to \ominfty R$
  (which is an unpointed
  map), we get to the subset
  $$(\iota_R)_* [S^k,\GL R]_* \subset [S^k,\ominfty R]$$
  of maps sending $S^k$ to the path component $1_{\pi_0 R}\in \pi_0 R$
  that are pointed at $1_{\pi_0 R}$;
  thus, compatibly with Remark \ref{rem-splitting-subgroup},
  $\iota_R$ realizes the inclusion of the subgroup
  $$\left( \wt R^0 S^k + 1_{\pi_0 R} \right)^\times$$
  in the multiplicative monoid of the ring $R^0 S^k$.
  If we now compose with the map $\bpshift$ intruduced in Remark
  \ref{sh-group-hom}, we get to pointed maps $[S^k,\ominfty R]_*$, or
  equivalently, to the group $\pi_k R \simeq \wt R^0 S^k \subset
  (R^0 S^k, \cdot_{R_0})$
  Our claim is that the map
  \begin{equation}
  \label{shifti}
      [S^k,\GL R]_* \xrightarrow{\textsc{Sh}\circ(\iota_R)_*}
          [S^k,\ominfty R]_*
  \end{equation}
  is a group isomorphism.
  By construction, $(\iota_R)_*$ is a group homomorphism into the
  multiplicative structure of its target, whereas by Remark
  \ref{sh-group-hom}, $\bpshift$ induces a homomorphism between the
  multiplicative structures of its source and target; hence
  it is sufficient to provide an inverse (as a set-map) to
  $\bpshift \circ (\iota_R)_*$.

  As (\ref{shifti}) is just a change of connected component, from
  $1_{\pi_0 R}$ to $0_{\pi_0 R}$, it is sufficient
  to realize the inverse change. To this end, let us consider the map
  \begin{displaymath}
      [S^k,\ominfty R]_* \xrightarrow{- + \mathbbm{1}}
          [S^k,\ominfty R]
  \end{displaymath}
  where $\mathbbm{1}$ denotes the map $S^k \to S^0 \xrightarrow{1_{\pi_0 R}}
  \ominfty R$, and the sum comes from the cogroup object structure on $S^k$.
  Under $- + \mathbbm{1}$, all the maps in $[S^k,\ominfty R]_*$
  are shifted to the path component of $1_{\pi_0 R} \in \pi_0 R$.
  By Remark \ref{GL-to-mult},
  each homotopy class in $[S^k,\ominfty R]_* + \mathbbm{1}$ determines a
  (unique, up to homotopy) map $S^k \to \GL R$, sending $S^k$ to the path
  component of $1_{\pi_0 R} \in \GL \pi_0 R$. In particular, the
  universal property of the pullback induces a map
  \begin{equation}
  \label{hotinvshifti}
      [S^k,\ominfty R]_* \longrightarrow
          [S^k,\ominfty R]_* + \mathbbm{1}
          \longrightarrow [S^k,\GL R]_*.
  \end{equation}
  As this composition has the effect of shifting from the connected component
  of $0_{\pi_0 R}$ to that of $1_{\pi_0 R}$, this is an inverse for
  (\ref{shifti}).
\end{proof}

\begin{rem}
  It follows from Remark \ref{rem-splitting-subgroup} that
  $(\wt R^0 S^k + 1_{\pi_0 R})^\times \simeq
    \wt R^0 S^k$.
  The proof of Proposition \ref{pigl} shows that, moreover, an isomorphism is
  given by $x + 1 \mapsto x$.
\end{rem}

\subsection{Square-zero extensions}\label{subs-sq-zero}
We will briefly recall the definitions of (trivial
and general) square-zero extensions of $\einf$-rings, referring the reader
to \cite[Chapter 7]{HA} for a detailed treatment.

\begin{rem}\label{triv-sq-zero-structure}
Let $A$ be an $\einf$-ring. In \cite[7.3.4.15]{HA} a functor
$$A \oplus - \colon \Mod_A \to \calg_{/A}$$
is constructed, sending an $A$-module $M$ to the \emph{trivial square-zero
extension of $A$ by $M$}. As explained in \cite[7.3.4.16]{HA} the notation
is motivated by the fact that it sends each object
$M \in \Mod_A$ to a commutative $\einf$-ring over $A$ whose
underlying spectrum is equivalent to $A \oplus M$ in $\Sp$.
The algebra structure on $A \oplus M$ is
``square-zero'' in the homotopy category of $\calg$, as the following
facts hold:
\begin{enumerate}
\item The unit map $\bbS \to A \oplus M$ is homotopic to the composition of
  the unit map $\bbS \to A$ with the inclusion $A \to A \oplus M$.
\item The multiplication
  $$(A \otimes A) \oplus (A \otimes M) \oplus (M \otimes A) \oplus
    (M \otimes M) \simeq (A \oplus M) \otimes (A \oplus M) \to A \oplus M$$
  is given as follows:
  \begin{enumerate}
  \item[$\bullet$] On $A\otimes A$, it is homotopic to the composition of the
    multiplication of $A$ with the inclusion $A \to A \oplus M$.
  \item[$\bullet$] On $A \otimes M$ and $M \otimes A$, it is given by
    composition of the action of $A$ on $M$ with the inclusion
    $M \to A \oplus M$.
  \item[$\bullet$] On $M \otimes M$ it is nullhomotopic.
  \end{enumerate}
\end{enumerate}
\end{rem}

\begin{rem}
  The functor $A \oplus -$ introduced above admits a left adjoint
  $$\calL_A \colon \calg_{/A} \to \Mod_A$$
  whose value on $A$ (with the identity as structure map) is denoted $L_A$
  and called the \emph{cotangent complex} of $A$.
\end{rem}

In order to describe the behavior of trivial square-zero extensions under
restriction of scalars, we need to recall the following result.

\begin{prop}\cite[7.3.4.14]{HA}\label{stabilization-of-algebras}
  Let $A$ be an $\einf$-ring. There is a canonical equivalence of \cats
  $$\Sp(\calg_{/A}) \simeq \Mod_A .$$
\end{prop}
\begin{proof}
  It follows from \cite[1.4.2.18]{HA} that $\Sp(\calg_{/A}) \simeq
  \Sp(\calg_{/A}^{A/})$. Moreover, by \cite[3.4.1.7]{HA}, there exists an
  equivalence $\calg^{A/} \simeq \calg(\Mod_A)$. Combining the two, we get:
  $$\Sp(\calg_{/A}) \simeq \Sp(\calg(\Mod_A)_{/A}).$$
  The objects in $\calg(\Mod_A)_{/A}$ are, by definition, equipped with a
  structure map with target $A$, in the \cat $\Mod_A$. Hence, taking fibers of
  the structure maps gives a functor $\calg(\Mod_A)_{/A}\to\Mod_A$, which,
  being left exact,
  in turn induces a functor
  $$
  \begin{tikzcd}
  {\mathrm{Exc}_*(\spaces_*^{\mathrm{fin}},\calg(\Mod_A)_{/A})} \ar[d, equal] \ar[r]&
  {\mathrm{Exc}_*(\spaces_*^{\mathrm{fin}},\Mod_A)} \ar[d, equal] \\
  \Sp(\calg(\Mod_A)_{/A})\ar[r] & \Sp(\Mod_A) \simeq \Mod_A
  \end{tikzcd}
  $$
  by pointwise composition
  (where the bottom right equivalence is due to the fact that $\Mod_A$
  is stable). To
  conclude, \cite[7.3.4.7]{HA} shows that this last functor is an equivalence.
\end{proof}

\begin{rem}\label{scalars-restriction-pullback}
  Let $f\colon A \to B$ be a morphism of $\einf$-rings. Then, there is an induced
  adjunction $\calg_{/A} \rightleftarrows \calg_{/B}$, where the left adjoint
  is given by postcomposition, and the right adjoint is given by pullback along
  $f$. Let us denote by $f^*\colon\calg_{/B}\to\calg_{/A}$ the right adjoint
  functor.
  The functor $f^*$ restricts to a functor
  $$F\colon \calg_{/B}^{B/} \to \calg_{/A}^{A/}$$
  by commutativity of the following diagram
  $$
  \begin{tikzcd}
  A \ar[rr] \ar[ddr, bend right, equal] \ar[dr, dotted]& &
    B \ar[ddr, bend right, equal] \ar[dr]&  \\
    & {R \times_B A} \ar[d] \ar[rr, crossing over] & & R \ar[d] \\
    & A \ar[rr] & & B.
  \end{tikzcd}
  $$
  The functor $F$, in turn, induces by pointwise composition the horizontal
  functors in the following diagram
  $$
  \begin{tikzcd}
  {\mathrm{Exc}_*(\spaces_*^{\mathrm{fin}},\calg_{/B}^{B/})} \ar[d, equal] \ar[r]&
    {\mathrm{Exc}_*(\spaces_*^{\mathrm{fin}},\calg_{/A}^{A/})} \ar[d, equal] \\
  {\Sp(\calg_{/B}^{B/})}\ar[r] \ar[d, "\simeq"] &
    {\Sp(\calg_{/A}^{A/})} \ar[d, "\simeq"] \\
  {\Sp(\calg(\Mod_B)_{/B})} \ar[r]\ar[d, equal]&\Sp(\calg(\Mod_A)_{/A})\ar[d, equal]\\
  {\mathrm{Exc}_*(\spaces_*^{\mathrm{fin}},\calg(\Mod_B)_{/B})}\ar[r]&
    {\mathrm{Exc}_*(\spaces_*^{\mathrm{fin}},\calg(\Mod_A)_{/A})}.
  \end{tikzcd}
  $$
  By inspection (recall that in Proposition \ref{stabilization-of-algebras}
  the equivalence
  $\calg(\Mod_A)_{/A} \simeq \Mod_A$ was induced by taking fibers of
  the structure maps), this functor is equivalent to the
  restriction of scalars functor $f^!\colon \Mod_B \to \Mod_A$.
  In particular, this implies that, given any $B$-module $M$, the $\einf$-rings
  $(B \oplus M) \times_B A$ and $A \oplus f^!M$ are equivalent, and that
  the square
  $$
  \begin{tikzcd}
  {A \oplus f^! M} \ar[r] \ar[d] & {B \oplus M} \ar[d] \\
  A \ar[r, "f"] & B
  \end{tikzcd}
  $$
  is Cartesian in the \cat $\calg$.
\end{rem}

\begin{defi}
  Let $R$ be an $\einf$-ring, and let $M$ be an $R$-module. We define a
  \emph{derivation} from $R$ to $M$ to be a map of $R$-modules
  $$L_R \to M.$$
\end{defi}

As it is clear, a derivation is equivalently determined (up to a contractible
space of choices) by its adjoint map $R \to R \oplus M$ (over $R$). If
$\eta\colon L_R \to M$ is a derivation, we will denote its adjoint map by
$d_\eta\colon R \to R \oplus M$.

\begin{defi}
  Let $R$ be an $\einf$-ring, and let $\eta\colon L_R \to M$ be a derivation from
  $R$ to an $R$-module $M$. Let $\phi\colon\wt R \to R$ be a morphism in $\calg$.
  We will say that $\phi$ is a \emph{square-zero extension} if there exists a
  Cartesian square
  $$
  \begin{tikzcd}
    \wt R \ar[r, "\phi"] \ar[d] & R \ar[d, "d_\eta"] \\
    R \ar[r, "d_0"] & R \oplus M
  \end{tikzcd}
  $$
  in the \cat $\calg$ (where $d_0$ is adjoint to the zero map $L_R \to M$).
  In this case, we will also say that $\wt R$ is a \emph{square-zero extension
  of $R$ by $M[-1]$}.
\end{defi}

\begin{rem}
  Differently from the classical case, being square-zero is not a property
  of an extension, but an additional structure.
  Given a square-zero extension $f\colon \wt R \to R$, in general, both the
  module $M$ and the derivation $\eta$ need not be uniquely determined, even
  up to equivalence.
  However, this happens in the most common situations (see
  \cite[7.4.1.26]{HA}).
\end{rem}

We conclude this section giving a characterization of the values of
$\calL_\bbS$ for monoid and group $\einf$-rings. This result will constitute
a key step for the first proof of Theorem \ref{big-thm}. A version of
Proposition
\ref{key-step} is proved in \cite[6.1]{BM05} using rigid models for the
categories of spectra. After setting up a few preliminary results,
we present a purely $\infty$-categorical proof of it.

\begin{prop}\label{key-step}
  The functor
  $$\calL_\bbS \circ \bbS[-] \colon \mon \to \Sp$$
  is naturally equivalent to $\Binf \circ (-)\gp$ (see Appendix \ref{prelim}).
  In particular, for any connective spectrum
  $M$, we have a natural equivalence
  $$\calL_\bbS(\bbS[\ominfty_a M]) \simeq M.$$
\end{prop}
In order to prove Proposition \ref{key-step}, we will first introduce the
analogue of the symmetric algebra functor, sending
a spectrum $X$ to the free $\einf$-ring $\Symm X$.
The free-forgetful adjunction $\Sp \rightleftarrows \calg$ thus obtained
is compatible with the free-forgetful adjunction $\spaces \rightleftarrows
\mon$ given by Proposition \ref{ladjs-factorization}, in the sense made
precise by Proposition \ref{free-sus-commute}. Moreover, the free functor
$\Symm$ will naturally
be augmented, and in its augmented fashion, will constitute a right inverse
functor for $\calL_\bbS$.

\begin{prop}\label{free-rings}
  There exists a free functor $$\Symm\colon \Sp \to \calg$$ left adjoint to the
  forgetful functor of Remark \ref{forget-algebras}.
\end{prop}
\begin{proof}
  This is an immediate consequence of \cite[3.1.3.5]{HA}.
\end{proof}

\begin{prop}\label{free-sus-commute}
  Given a space $X$, there exists a natural equivalence
  $$\bbS[\free_+ X] \simeq \Symm \susinftyp X$$
  where $\free_+$ is the functor defined in Proposition
  \ref{ladjs-factorization}.
\end{prop}
\begin{proof}
  Let us consider the following diagram
  $$
  \begin{tikzcd}[column sep = huge, row sep = huge]
    \spaces \ar[r, shift left=1.1ex, "\free_+"]
      \ar[d, shift right=1.1ex, "\susinftyp"']
      \ar[d, phantom, "\scalebox{0.7}{$\dashv$}" description]
      \ar[r, phantom, "\scalebox{0.7}{\rotatebox{-90}{$\dashv$}}" description]&
    \mon \ar[l, shift left=1.1ex] \ar[d, shift right=1.1ex, "\text{$\bbS[-]$}"']
      \ar[d, phantom, "\scalebox{0.7}{$\dashv$}" description]\\
  \Sp \ar[r, shift left=1.1ex, "\Symm"]
    \ar[u, shift right=1.1ex, "\ominfty\vphantom{\susinftyp}"']
    \ar[r, phantom, "\scalebox{0.7}{\rotatebox{-90}{$\dashv$}}" description]&
  \calg. \ar[l, shift left=1.1ex] \ar[u, shift right=1.1ex, "\ominfty_m"']
  \end{tikzcd}
  $$
  We can reformulate what we want to prove by saying that the square determined
  by the left adjoints commute. Since the square determined by the right
  adjoints clearly commutes, we are done.
\end{proof}

Similarly to what happens for $\bbS[-]$ (see Remark
\ref{natural-augmentation}), the free-forgetful adjunction $\Sp
\rightleftarrows \calg$ factors through
the \cat $\calg_{/\bbS}$ as a composite adjunction
$$\Sp \rightleftarrows \calg_{/\bbS} \rightleftarrows \calg.$$

\begin{prop}\label{free-aug-adj}\cite[7.3.4.5]{HA}
  There exists an adjunction
  $$\sym\aug^* : \Sp \rightleftarrows \calg_{/\bbS} : \scrI$$
  with the following properties:
  \begin{enumerate}
  \item The functor $\sym\aug^*$ is given by composition
    $$\Sp \simeq \Sp_{/0} \xto{\Symm} \calg_{/\bbS}.$$
  \item The functor $\scrI$ is given by taking fibers of the structure maps
    in $\calg_{/\bbS}$.
  \item The composition
    $$\Sp \xto{\sym\aug^*} \calg_{/\bbS} \xto{\scrI} \Sp$$
    is equivalent to $X \mapsto \prod_{n>0} \sym^n (X)$ (see
    \cite[3.1.3.9]{HA}).
  \end{enumerate}
\end{prop}

We will often abuse notation, and denote $\sym\aug^*$ just by $\Symm$.

\begin{prop}\label{L-sym-inverses}
  The composition $\calL_\bbS \Symm$ is naturally equivalent to the identity
  functor.
\end{prop}
\begin{proof}
  Let $X$ and $Y$ be spectra. We have a chain of equivalences
  \begin{displaymath}
  \begin{split}
  \Map_\Sp(\calL_\bbS \Symm X, Y)
    &\simeq \Map_{\calg_{/\bbS}}(\Symm X, \bbS\oplus Y)\\
    &\simeq \Map_\Sp(X, \scrI(\bbS\oplus Y))\\
    &\simeq \Map_\Sp(X,Y)
  \end{split}
  \end{displaymath}
  from which we deduce that $\calL_\bbS \Symm$ and $\id_\Sp$ represent the
  same functor in the \cat $\Sp$, and are therefore naturally
  equivalent.
\end{proof}

We are now ready to prove Proposition \ref{key-step}.

\begin{proof}[Proof of Proposition \ref{key-step}]
  By \cite[4.9]{ggn15}, we have a natural equivalence
  $$\Fun^{\mathrm{L}}(\mon,\Sp) \stackrel{\sim}{\longrightarrow} \Sp$$
  given by evaluation at $\free_+ (\pt)$, the free $\einf$-space generated by
  one point.
  By what we have seen so far, we have that
  \begin{displaymath}
  \begin{split}
    \calL_\bbS (\bbS[\free_+ (\pt)])
      &\stackrel{\vphantom{(}^{(\ref{free-sus-commute})}}{\simeq}
        \calL_\bbS (\Symm \susinftyp (\pt))\\
    &\stackrel{\vphantom{(}^{(\ref{L-sym-inverses})}}{\simeq} \susinftyp(\pt)\\
    &\stackrel{\vphantom{(}^{(\ref{ladjs-factorization})}}{\simeq}
      \Binf \left(\free_+ (\pt)\right)\gp
  \end{split}
  \end{displaymath}
  hence $\calL_\bbS \circ \bbS[-]$ and $\Binf \circ (-)\gp$ are naturally
  equivalent.
  The second part of the statement just follows from
  the equivalence between $\einf$-groups and connective
  spectra (see Remark \ref{infinite-loops}) and the
  fact that $(-)\gp$ is the identity on $\einf$-groups, as it is left adjoint to
  the inclusion $\grp \to \mon$.
\end{proof}

\subsection{Groups of units of square-zero extensions, the setup}
\label{subs-setup}
Our goal in this and the following two sections is to give two proofs of
the following generalization of
Proposition \ref{classicalresult} in our homotopy coherent setting.
\begin{thm}\label{big-thm}
  Let $R$ be a connective $\einf$-ring, and let $\wt R \to R$ be a square-zero
  extension by a connective $R$-module $M$. By applying $\gl$ to
  $\wt R \to R$, we obtain a map of spectra
  $$\phi\colon \gl\wt R \to \gl R.$$
  The fiber of $\phi$ is naturally equivalent to $M$ in the \cat $\Sp$.
\end{thm}

\begin{rem}\label{connectivity-hp}
  As $\gl$ (introduced in Notation \ref{notat-gl}) lands, by
  definition, in the image of $\Binf$ (defined in Proposition
  \ref{ladjs-factorization}), upon its application, all nonconnective
  information is lost (see Remark \ref{infinite-loops}).
  Hence, we can restrict ourselves to work with connective $\einf$-rings
  and connective modules without loss of generality, and suitably replace the
  relevant objects with their connective covers when dealing with nonconnective
  ones.
\end{rem}

Our strategy for the proof will be the following:
in this section, we will first show how to reduce the problem
from general square-zero extensions
to trivial ones, and then how to reduce it further to
trivial square-zero extensions of the sphere spectrum $\bbS$.
In the following two sections, we will show that the
theorem indeed holds in this last case. In Section \ref{subs-top-down}, we
will do the
last step in an exquisitely ``higher algebraic'' fashion, by giving a proof
entirely at the level of spectra, and in particular recovering the ordinary
result as a particular case. In Section \ref{subs-bottom-up}, we will give an
alternative
proof substantially founded on the space level, extending the homotopy
coherent result from the ordinary one.

\begin{prop}\label{one-plus-m}
  Let $R$ be a connective $\einf$-ring, and let $M$ be any connective
  $R$-module. Let
  $$\gl(R \oplus M) \to \gl R$$
  be the map obtained by applying the functor $\gl$ to the trivial square-zero
  extension $R \oplus M \to R$ and let $1+M$ denote the fiber of
  $\gl(R \oplus M) \to \gl R$. Then given any square-zero extension $\wt R \to
  R$ of $R$ by $M$, the fiber of the induced map
  $$\gl\wt R \to \gl R$$
  is naturally equivalent to $1+M$.
\end{prop}
\begin{proof}
  First of all, we observe that $\gl(R \oplus M) \simeq \gl R \oplus (1+M)$.
  In fact, since the map of $\einf$-rings $R\oplus M \to R$ admits a section,
  the same is true for its
  image under the right adjoint functor $\gl$, and hence the co/fiber sequence
  $1+M \to \gl(R\oplus M) \to \gl R$ splits.
  In particular, $1+M \simeq \coker \left(\gl R \to \gl (R \oplus M)\right)$,
  so that
  $$(1+M)[1] \simeq 1+(M[1]).$$
  By virtue of this canonical identification, we will unambiguously just write
  $$1+M[1].$$

  Let us now suppose that $\wt R \to R$ is a square-zero extension of
  $R$ by $M$. By definition, $\wt R$
  sits in a Cartesian square
  $$
  \begin{tikzcd}
    \wt R \ar[d] \ar[r] & R \ar[d] \\
    R \ar[r] & R \oplus M[1]
  \end{tikzcd}
  $$
  in the \cat $\calg$. Upon applying the functor $\gl$, we obtain
  the co/Cartesian square
  $$
  \begin{tikzcd}
    \gl\wt R \ar[d] \ar[r] & \gl R \ar[d] \\
    \gl R \ar[r] & \gl R \oplus (1+M[1])
  \end{tikzcd}
  $$
  in the \cat $\Sp$.

  The result now follows from the pasting law of pushouts
  applied to the following diagram
  $$
  \begin{tikzcd}
    \gl\wt R \ar[d] \ar[r] & \gl R \ar[d] \ar[r] & 0 \ar[d] \\
    \gl R \ar[r] & \gl R \oplus (1+M[1]) \ar[r] & 1+M[1].
  \end{tikzcd}
  $$
\end{proof}

Hence, we are reduced to prove our result in the case of trivial square-zero
extensions. Next step will be to reduce to square-zero extensions of $\bbS$.

\begin{lemma}\label{reduction}
Let $R$ be an $\einf$-ring, and let $R \oplus M$ be a trivial square-zero
extension by an $R$-module $M$. Then, if Theorem \ref{big-thm} holds for
trivial square-zero extensions of $\bbS$, it holds for $R\oplus M \to R$.
\end{lemma}
\begin{proof}
Let $R$ be an $\einf$-ring, and let $\bbS \to R$ be the unit morphism. Given
any $R$-module $M$, this morphism induces by Remark
\ref{scalars-restriction-pullback} a Cartesian square
\begin{equation}\label{general-split}
\begin{tikzcd}
  \bbS \oplus M \ar[r] \ar[d] & R \oplus M \ar[d] \\
  \bbS \ar[r] & R
\end{tikzcd}
\end{equation}
in the \cat $\calg$.
The functor $\gl$, being right adjoint, sends the square
(\ref{general-split}) to a co/Cartesian square in $\Sp$:
\begin{equation}\label{particular-split}
\begin{tikzcd}
\gl \left( \bbS \oplus M \right) \ar[r] \ar[d] & \gl \left( R \oplus M \right)
  \ar[d] \\
\gl \bbS \ar[r] & \gl R.
\end{tikzcd}
\end{equation}
Let us denote by $P$ the fiber of $\gl(R\oplus M) \to \gl R$. Then, $P$ fits
into the diagram:
$$
\begin{tikzcd}
P \ar[r] \ar[d] & \gl (\bbS \oplus M) \ar[r] \ar[d] & \gl (R \oplus M) \ar[d]\\
0 \ar[r] & \gl \bbS \ar[r] & \gl R
\end{tikzcd}
$$
where the map $P \to \gl(\bbS\oplus M)$ is induced by the fact that
(\ref{particular-split}) is Cartesian.
As the outer square is Cartesian by construction, it follows from the pasting
law that $P$ is canonically equivalent to the fiber of $ \gl(\bbS\oplus M) \to
\gl\bbS$, which by hypothesis is equivalent to $M$.
\end{proof}

\begin{lemma}\label{bunch-of-adjs}
  Given any spectrum $M$, let $1+M$ denote the fiber of the map
  $$\gl(\bbS\oplus M) \to \gl\bbS$$
  obtained by applying $\gl$ to the square-zero extension $\bbS\oplus M \to
  \bbS$.
  Then, the functor $$M \mapsto 1+M$$
  admits a left adjoint.
\end{lemma}
\begin{proof}
  For any spectrum $X$, giving a map $X \to 1+M$ is
  equivalent to give a map $X \to \gl(\bbS\oplus M)$
  together with a nullhomotopy for its postcomposition
  with the map $\gl(\bbS\oplus M) \to \gl\bbS$.
  Upon passing to adjoints we see that, by virtue of the explicit
  description of the augmentation map of $\bbS[-]$ given in Remark
  \ref{explicit-augmentation}, it is equivalent to give a map
  $$
  \begin{tikzcd}
    \bbS[\ominfty_a X] \ar[rr] \ar[dr, bend right] & &
      \bbS \oplus M \ar[dl, bend left] \\
    & \bbS &
  \end{tikzcd}
  $$
  that is, a morphism in $\Map_{\calg_{/\bbS}}(\bbS[\ominfty_a X],
  \bbS\oplus M)$. We can rephrase what we just observed more precisely,
  by saying that the functor $M \mapsto 1+M$ fits in the following
  diagram
  \begin{equation}\label{diagrammone}
  \begin{tikzcd}[column sep = huge, row sep = huge]
  \calg_{/\bbS}
  \ar[r, shift left=1.1ex, "\calL_\bbS"]
  \ar[r, phantom, "\scalebox{0.7}{\rotatebox{-90}{$\dashv$}}" description]
  \ar[d, phantom, "\scalebox{0.7}{$\dashv$}" description]
  \ar[d, shift left=1.1ex, "\Binf\scrG\vphantom{\bbS[\ominfty_a -]}"]&
  \Sp
  \ar[l, shift left=1.1ex, "\bbS\oplus -"]
  \ar[dl, bend left, phantom, "\scalebox{0.7}{\rotatebox{130}{$\dashv$}}" description]
  \ar[dl, bend left, shift right=1.1ex, "1+ "' pos=0.4]\\
  \Sp
  \ar[u, shift left=1.1ex, "{\bbS[\ominfty_a -]}"]
  \ar[ur, shift right=1.1ex, bend right]&
  \end{tikzcd}
  \end{equation}
  i.e. if we denote by $\scrG$ the right adjoint to the factorization of
  $\bbS[\ominfty_a -]$ through $\calg_{/\bbS}$ given by Remark
  \ref{natural-augmentation}, then $M\mapsto 1+M$ is given by the composite
  right adjoint $\Binf\scrG(\bbS\oplus -)$.
  Hence, the functor $M \mapsto 1+M$ is right adjoint to the functor
  $X \mapsto \calL_\bbS ( \bbS[\ominfty_a X])$.
\end{proof}

\begin{rem}
  We stress that in (\ref{diagrammone})
  both $\bbS[\ominfty_a -]$ and $\Binf\scrG$ are functors from the
  \emph{slice \cat} $\calg_{/\bbS}$; in particular, it is important to
  make a distinction between $\Binf\scrG$ and $\gl$.
\end{rem}

\subsection{Groups of units of square-zero extensions, top-down}
\label{subs-top-down}

We are now ready to give a first proof.

\begin{proof}[Proof of Theorem \ref{big-thm}]
  By virtue of Proposition \ref{one-plus-m} and Lemma \ref{reduction}, all it
  is left to do, it is to prove the theorem for trivial square-zero extensions
  $$\bbS \oplus M \to \bbS$$
  where $M$ is a connective spectrum.
  As we want to show that $1+M$ is naturally equivalent to $M$, it is
  enough to prove that the functor $X\mapsto\calL_\bbS (\bbS[\ominfty_a
  X])$, left adjoint to $M \mapsto 1+M$, is naturally equivalent to the
  identity functor.
  But, by virtue of Proposition \ref{key-step}, we know that this is indeed
  the case.
\end{proof}

\subsection{Groups of units of square-zero extensions, bottom-up}
\label{subs-bottom-up}

It is possible to give an alternative proof of (the last step of
the proof of) Theorem \ref{big-thm},
based on space level arguments
and Proposition \ref{classicalresult}.
We dedicate this section to give such an alternative proof. The idea is to
show the equivalence of the spaces $\ominfty M$ and $\ominfty (1+M)$, and
then to show there are no obstructions to lift the comparison map
at the level of connective spectra.
We begin with the following observation.

\begin{prop}
  Let $R$ be a connective $\einf$-ring, and let $\wt R \to R$ be a
  square-zero extension of $R$ by a connective $R$-module $M$. Let us denote by
  $1+M$ the fiber of the induced map $\gl\wt R \to \gl R$. Then, the
  spaces $\ominfty (1+M)$ and $\ominfty M$ are equivalent in the \cat
  $\spaces$ of spaces.
\end{prop}
\begin{proof}
  By applying $\ominfty$ to the co/fiber sequence
  $$1+M \to \gl\wt R \to \gl R$$
  we get a fiber sequence of pointed spaces
  $$\ominfty (1+M) \to \GL \wt R \to \GL R.$$
  By the commutativity of the following diagram in the \cat $\spaces_*$
  $$
  \begin{tikzcd}
    \ominfty (1+M) \ar[rr] \ar[dd] &&
        \GL \widetilde R \ar[dd]
        \ar[dr, "\textsc{Sh}(\iota_{\widetilde R})"]&\\
    & \ominfty M \ar[rr, crossing over] & &
        \ominfty \widetilde R \ar[dd]\\
    \pt \ar[rr] \ar[dr]& &
      \GL R \ar[dr, "\textsc{Sh}(\iota_{R})"]&\\
    & \pt \ar[rr] \ar[from=uu, crossing over]& & \ominfty R\\
  \end{tikzcd}
  $$
  (where $\bpshift$ denotes the basepoint shift of Definition
  \ref{def-bpshift}, and $\iota_R$ denotes the map given in Remark
  \ref{GL-to-mult})
  we get a pointed map $\psi\colon\ominfty (1+M) \to \ominfty M$. It follows from
  Proposition \ref{pigl} that we have the following induced morphism between
  the homotopy fiber exact sequences
  \begin{equation}
  \label{hotfibles}
  \begin{tikzcd}[column sep=tiny]
    \cdots {\pi_k \GL \widetilde R} \ar[r] \ar[d, "\simeq"]
      & {\pi_k \GL R} \ar[r] \ar[d, "\simeq"]
      & {\pi_{k-1} (1+M)} \ar[r] \ar[d]
      & {\pi_{k-1} \GL \widetilde R} \ar[r] \ar[d, "\simeq"]
      & {\pi_{k-1} \GL R \cdots} \ar[d, "\simeq"] \\
  \cdots \pi_k \widetilde R \ar[r]
      & {\pi_k R} \ar[r]
      & {\pi_{k-1} M} \ar[r]
      & {\pi_{k-1} \widetilde R}  \ar[r]
      & {\pi_{k-1} R \cdots}
  \end{tikzcd}
  \end{equation}
  which, by the five lemma, induces isomorphisms $\pi_k (1+M) \simeq \pi_k M$
  for $k \geq 1$.
  To conclude, wee need to prove that $\pi_0 (\psi)\colon \pi_0 (1+M) \to
  \pi_0 M$ is an isomorphism. It
  follows from the exactness of the sequence
  $$
  \begin{tikzcd}
  \cdots \ar[r] & \pi_1 \widetilde R \ar[r]
      & \pi_1 R \ar[r]
      & \pi_0 M \ar[r]
      & \pi_0 \widetilde R \ar[r]
      & \pi_0 R \ar[r]& 0
  \end{tikzcd}
  $$
  that the map $\pi_1 R \to \pi_0 M$ is actually the zero map. If we consider
  (\ref{hotfibles}) in the case $k=1$, we get
  $$
  \begin{tikzcd}[column sep = 1.5em]
    \cdots \pi_1 \widetilde R \ar[r] \ar[d, "\simeq"]
      & \pi_1 R \ar[r] \ar[d, "\simeq"]
      & \pi_0 (1+M) \ar[r] \ar[d]
      & (\pi_0 \widetilde R)^{\times} \ar[r] \ar[d, hook]
      & \left(\pi_0 R\right)^{\times} \ar[r] \ar[d, hook] & 0 \\
  \cdots \pi_k \widetilde R \ar[r]
      & \pi_1 R \ar[r, "0"]
      & \pi_0 M \ar[r]
      & \pi_0 \widetilde R \ar[r]
      & \pi_0 R \ar[r]& 0.
  \end{tikzcd}
  $$
  The five-lemma implies that $\pi_0 (1+M)\to\pi_0 M$ is a monomorphism, which
  in turn tells us that the map $\pi_1 R\to\pi_0 (1+M)$ is also the zero map.
  Hence, $\pi_0 (1+M)$ is the kernel of $$(\pi_0 \widetilde R)^{\times}
  \to \left(\pi_0 R\right)^{\times}$$ and by Proposition \ref{classicalresult},
  $\pi_0 (1+M) \simeq \pi_0 M$, where one isomorphism is given by
  $u \mapsto u-1$.
  Finally, it follows from the definition of $\bpshift$ that
  the map $\pi_0 (\psi)\colon \pi_0 (1+M) \to \pi_0 M$
  is exactly given by $u \mapsto u-1$.
\end{proof}

We can specialize the previous proposition to the case of trivial square-zero
extensions of $\bbS$, obtaining the following corollary.

\begin{cor}\label{ominf-oneplus}
  The functor $\ominfty \circ (1+ -)\colon\Sp^{cn} \to \spaces$ is
  equivalent to the functor $\ominfty$.
\end{cor}

We conclude this section with the promised second proof.

\begin{proof}[(Alternative) Proof of Theorem \ref{big-thm}]
  Again, keeping in mind Proposition \ref{one-plus-m} and Lemma
  \ref{reduction}, we just have to show that the functor $1+ -$ described in
  Lemma \ref{bunch-of-adjs} is equivalent to the identity functor.
  It follows from \cite[2.10]{ggn15} and the equivalence $\Sp^{cn} \simeq
  \grp$ that there is an equivalence
  $$\Fun^{\Pi} (\Sp^{cn},\Sp^{cn}) \simeq \Fun^{\Pi} (\Sp^{cn},\spaces)$$
  (where $\Fun^{\Pi}(\scrC,\scrD)$ denotes the full subcategory of
  $\Fun(\scrC,\scrD)$ spanned by product preserving functors) given by
  postcomposition with $\ominfty$. By Lemma \ref{bunch-of-adjs}, the functor
  $1+ -$ preserves all limits (in particular, products), and by Corollary
  \ref{ominf-oneplus}, its postcomposition with $\ominfty$ is equivalent
  to $\ominfty$ itself. As this is also true for the identity functor,
  $1+ -$ and $\id_{\Sp^{cn}}$ must be equivalent.
\end{proof}

\section{Logarithmic derivatives of $\einf$-rings}\label{applications}

Let $R$ be an ordinary commutative ring, and let $M$ be an $R$-module.
Given a derivation $\partial\colon R \to M$, the function
\begin{displaymath}
\begin{split}
  \log_\partial\colon R^\times &\to M \\
  r &\mapsto \partial(r)r^{-1}
\end{split}
\end{displaymath}
is easily seen to be a group homomorphism (from the group of units of $R$ to
the underlying additive Abelian group of $M$), and it is called the
\emph{logarithmic derivative} relative to $\partial$.
Our goal for this section, is to construct a homotopy coherent analogue of
logarithmic derivatives for $\einf$-rings.

To this end, let now $R$ be a connective $\einf$-ring (see Remark
\ref{connectivity-hp}).
and let $\partial\colon L_R \to M$ be a derivation from $R$ to an $R$-module
$M$.
Let us denote by $\wt\partial\colon R \to M$ the composite map
$$R \xto{d_\partial} R\oplus M \xto{\mathrm{pr}_2} M$$
in the \cat $\Sp$ (where $d_\partial$ is the map adjoint to $\partial$, and
$\mathrm{pr}_2$ is the projection on the module part of the square-zero
extension $R \oplus M$). As a first approach, it is possible
to mimic in a straightforward way the ordinary definition at the level of
spaces, and thus to consider the composition
\begin{equation}\label{naive-derivative}
  \GL R \xto{\Delta} \GL R \times \GL R \xto{(\iota_R \circ \inv) \times
    (\ominfty\wt\partial \circ \iota_R)} \ominfty R \times \ominfty M
  \xto{a} \ominfty M
\end{equation}
in the \cat of spaces
(where $\iota_R$ is as in Remark \ref{GL-to-mult},
$\inv$ is the inversion map given by the $\einf$-group structure of $\GL R$,
and $a$ is given by the action of $R$ on $M$).
Such a map can be shown to be a morphism of group objects in the homotopy
category $\ho\spaces$ (i.e. a map of H-spaces). Our goal
will be to promote
such a map to a map of $\einf$-groups, or to be more precise, to produce
a morphism of connective spectra whose underlying map is homotopic to
(\ref{naive-derivative}). In order to do so, we will exploit Theorem
\ref{big-thm}, applied to the trivial square-zero extension $R \oplus M$.

\begin{constr}\label{constr-log-der}
  Let $R$ be a connective $\einf$-ring, and let $\partial\colon L_R \to M$ be
  a derivation of $R$ into an $R$-module $M$.
  Theorem \ref{big-thm} implies that there exists a co/fiber sequence
  $$M \to \gl (R \oplus M) \to \gl R$$
  in the \cat $\Sp$, which splits, since $R \oplus M$ is a trivial square-zero
  extension; that is, we have that
  \begin{equation}\label{gl-plus-splitting}
  \gl(R \oplus M) \simeq (\gl R ) \oplus M
  \end{equation}
  in the \cat $\Sp$.
  Let us momentarily denote by $i\colon\gl R\to\gl (R\oplus M)$ the direct
  summand inclusion.
  Let now
  $$d_\partial \colon R \to R\oplus M$$ be the map adjoint to $\partial$ (which,
  by definition, is a section of the trivial square-zero extension $R \oplus M
  \to R$).
  We will denote by $\pre\log_\partial \colon \gl R \to \gl (R \oplus M)$, the
  composition
  $$
  \begin{tikzcd}[column sep = huge, row sep = huge]
    \gl R \ar[r, "\Delta"]
      \ar[ddr, bend right=45, "\pre\log_\partial"']
    & \gl R \oplus \gl R
      \ar[d, "(i \circ \inv) \oplus \gl(d_\partial)"] \\
    & \gl (R \oplus M) \oplus \gl (R \oplus M) \ar[d, "\mu"] \\
    & \gl (R \oplus M)
  \end{tikzcd}
  $$
  (where $\mu$ is the multiplication map induced by Remark
  \ref{additive-spectra}).
\end{constr}

\begin{rem}\label{plog-to-zero}
  Let us denote by $p\colon\gl(R\oplus{M}) \to \gl R$ the projection to the
  direct summand given by (\ref{gl-plus-splitting}).
  The composition
  $$\gl R \xto{\pre\log_\partial} \gl(R\oplus{M}) \xto{p} \gl R$$
  is nullhomotopic. In fact, by construction, the triangle
  $$
  \begin{tikzcd}
    \gl R \ar[rr, "i \circ \inv"] \ar[dr, bend right, "\inv"']& &
      \gl(R\oplus{M}) \ar[dl, bend left, "p"] \\
    & \gl R &
  \end{tikzcd}
  $$
  commutes.
  On the other hand, by
  definition of $d_\partial$, the map $\gl(d_\partial)$ fits into the
  commutative triangle
  $$
  \begin{tikzcd}
    \gl R \ar[rr, "\gl(d_\partial)"] \ar[dr, bend right, equal] & &
      \gl(R\oplus{M}) \ar[dl, bend left, "p"] \\
    & \gl R. &
  \end{tikzcd}
  $$
  Since, as it follows e.g. from Remark \ref{additive-spectra},
  $\mu \circ (p\oplus{p}) \simeq p \circ \mu$, we have that
  \begin{displaymath}
  \begin{split}
  p \circ \pre\log_\partial &=
    p \circ \mu \circ (i \circ \inv \oplus \gl(d_\partial))\circ\Delta\\
  &\simeq\mu \circ (p\oplus{p}) \circ
    (i \circ \inv \oplus \gl(d_\partial))\circ\Delta\\
  &\simeq \mu \circ (p \circ i \circ \inv) \oplus (p \circ \gl(d_\partial))
    \circ\Delta\\
  &\simeq \mu \circ (\inv \oplus \id)\circ\Delta
  \end{split}
  \end{displaymath}
  which, again by Remark \ref{additive-spectra}, is nullhomotopic.
\end{rem}

\begin{defi}
  Let $R$ be a connective $\einf$-ring, and $\partial \colon L_R \to M$ a
  derivation. We define the \emph{logarithmic derivative} $\log_\partial$
  induced by $\partial$ as the map of spectra
  $$\log_\partial \colon \gl R \to M$$
  induced by the map $\pre\log_\partial$ given by Construction
  \ref{constr-log-der},
  and the universal property of fibers, applied by virtue of Remark
  \ref{plog-to-zero}; i.e. $\log_\partial$ is the essentially
  unique map rendering the following diagram commutative
  $$
  \begin{tikzcd}
    \gl R \ar[d, dotted, "\log_\partial"']
    \ar[dr, bend left, "\pre\log_\partial"]&& \\
  M \ar[r] & \gl (R \oplus M) \ar[r] & \gl R.
  \end{tikzcd}
  $$
\end{defi}

\begin{rem}
  By virtue of Remark \ref{infinite-loops}, upon applying $\ominfty_a$ to the
  logarithmic derivative, we obtain a map of $\einf$-groups
  $$\ominfty_a \log_\partial \colon \GL R \to \ominfty_a M$$
  whose underlying map of spaces is homotopic to the map
  (\ref{naive-derivative}) given in the introduction to this chapter.

  To see this, let us begin by applying the natural transformation
  $\iota\colon \GL \Rightarrow \ominfty_m$ given in Remark \ref{GL-to-mult}
  to $d_\partial$, in order to get a commutative square
  $$
  \begin{tikzcd}
  \GL R \ar[r, "\iota_R"] \ar[d, "\GL (d_\partial)"'] &
    \ominfty_m R \ar[d, "\ominfty_m (d_\partial)"] \\
  \GL (R \oplus M) \ar[r, "\iota_{R\oplus R}"] & \ominfty_m (R \oplus M)
  \end{tikzcd}
  $$
  in the \cat of $\einf$-groups. If we consider the above square in
  the \cat $\spaces$,
  combined with (\ref{gl-plus-splitting}), we get the following
  $$
  \begin{tikzcd}
  \GL R \ar[r, "\iota_R"] \ar[d, "\GL (d_\partial)"'] &
    \ominfty R \ar[d, "\ominfty (d_\partial)"] \\
  \GL R \times \ominfty M \ar[r, "\iota_{R\oplus M}"] \ar[d, "\pr_2"'] &
    \ominfty R \times \ominfty M \ar[d, "\pr_2"] \\
  \ominfty M \ar[r, equal] & \ominfty M
  \end{tikzcd}
  $$
  (where $\mathrm{pr}_2$ is the obvious projection)
  showing that
  \begin{equation}\label{pr-GL}
  \ominfty \wt \partial \circ \iota_R \simeq \pr_2 \circ
  \ominfty (d_\partial) \circ \iota_r \simeq \pr_2 \circ \GL (d_\partial)
  \end{equation}
  as map of spaces. Now, by (\ref{gl-plus-splitting}), $\ominfty_a
  \log_\partial$ is homotopic to
  $$\GL R \xto{\ominfty_a (\pre\log_\partial)} \GL (R\oplus M)
    \xto{\mathrm{pr}_2} \ominfty_a M.$$
  Unraveling the definitions, we have that, in the \cat $\spaces$:
  \begin{displaymath}
  \begin{split}
  \ominfty \log_\partial &\simeq \pr_2 \circ \ominfty(\pre\log_\partial)\\
  &\simeq
  \pr_2 \circ \ominfty \Big(\mu \circ (i \circ \inv)\oplus\gl(d_\partial)
    \circ \Delta\Big) \\
  &\simeq \pr_2 \circ \wt m \circ \ominfty \Big(i \circ \inv \oplus
    \gl(d_\partial)\Big) \circ \Delta \\
  &\simeq \pr_2 \circ \wt m \circ \Big(\ominfty (i \circ \inv) \times
    \GL(d_\partial)\Big) \circ \Delta \\
  \end{split}
  \end{displaymath}
  (where $\wt m$ denotes the multiplication map on $\GL (R \oplus M)$).
  By Remark
  \ref{triv-sq-zero-structure}, this is homotopic to
  $$m \circ \Big( (\iota_R \circ \inv) \times (\pr_2 \circ \GL(d_\partial))
  \Big) \circ \Delta$$
  which in turn, by (\ref{pr-GL}), is homotopic to (\ref{naive-derivative}).
\end{rem}

We conclude this chapter showing that for ordinary rings, regarded as discrete
$\einf$-rings, our definition of logarithmic derivatives recovers the usual
one, upon passing to connected components.

\begin{rem}
  If $R$ is a discrete $\einf$-ring, \cite[7.4.3.8]{HA} shows that
  $$\pi_0 L_R \simeq \Omega_{\pi_0 R}$$
  in the ordinary category of discrete $\pi_0 R$-modules.
  As a consequence, any derivation $\partial\colon L_R \to M$
  determines an ordinary derivation $\pi_0 R \to \pi_0 M$.
\end{rem}

\begin{prop}
  Let $R$ be a discrete $\einf$-ring, and let $\partial\colon L_R \to
  M$ be a derivation of $R$ into an $R$-module $M$. Then, the morphism
  $$\pi_0\log_\partial\colon (\pi_0 R)^\times \to \pi_0 M$$
  is the ordinary logarithmic derivative associated to $\pi_0 \partial$.
\end{prop}
\begin{proof}
  Let us denote by $\wt \partial \colon \pi_0 R \to \pi_0 M$ the ordinary
  derivation determined by $\pi_0 \partial\colon \Omega_{\pi_0 R} \to \pi_0 M$.
  Unraveling the definitions, we see that the value of
  $$\pi_0 \pre\log_\partial\colon (\pi_0 R)^\times \to \pi_0(\gl (R
  \oplus M)) \simeq (\pi_0 R)^\times \oplus \pi_0 M$$ on any element $r \in
  \left(\pi_0 R\right)^{\times}$ is
  \begin{displaymath}
  \begin{split}
  \pi_0 \pre\log_\partial (r)
    &=\pi_0 \Big(\mu \circ \big(i \circ
      \inv \oplus \gl(d_\partial)\big)\circ\Delta\Big)(r)\\
    &=\pi_0 \Big(\mu \circ (i \circ \inv \oplus \gl(d_\partial))\Big)(r,r)\\
    &=\pi_0 (\mu)\Big(\big(r^{-1},0\big),\big(r,\wt\partial(r)\big)\Big)\\
    &= \Big(r^{-1}, 0\Big)\cdot\Big(r,\wt\partial (r)\Big)\\
    &= \Big(1, \wt\partial (r)r^{-1}\Big)
  \end{split}
  \end{displaymath}
  where the last equality follows from Remark \ref{triv-sq-zero-structure}.
  Now, it follows from Proposition \ref{classicalresult} that the induced
  map $\pi_0\log_\partial\colon (\pi_0 R)^\times \to \pi_0 M$ is just the
  projection on the second factor of $\pi_0\pre\log_\partial$, and hence it
  is given on elements by
  $$r \mapsto \wt\partial(r)r^{-1}.$$
\end{proof}

\appendix
\section{The classical result}\label{appendix-classical}
Our goal in this appendix is to recall and prove Proposition
\ref{classicalresult}, which is the ordinary version of Theorem \ref{big-thm}.
Throughout by ``ring'' we mean commutative ring with unit;
by ``ring (homo)morphism'' we mean unit-preserving ring homomorphism.

\begin{defi}
    Given a surjective morphism of rings
    \begin{displaymath}
        \phi\colon \widetilde R \to R
    \end{displaymath}
    the morphism $\phi$ is said to be a \emph{square-zero extension} if
    $\left ( \ker \phi \right )^2 = 0$.
\end{defi}

With a little abuse of terminology, we will say that ``$\wt R$ is a
square-zero extension of $R$ by $I$'' if $R$ and $\wt R$ are rings, and
$R \simeq \wt R/I$ for some ideal
$I \subset R$ such that $I^2 = 0$.

\begin{rem}
    Let $\phi\colon \widetilde R \to R$ be a square-zero extension and let
    $I\coloneqq\ker\phi$ denote its kernel.
    Then we have a short exact sequence of
    $\wt R$-modules
    \begin{displaymath}
        0 \longrightarrow I \longrightarrow \widetilde R \longrightarrow R
        \longrightarrow 0.
    \end{displaymath}
\end{rem}

\begin{prop}
\label{classicalresult}
    Given a square-zero extension $\phi\colon \widetilde R \to R$, let
    $I\coloneqq\ker\phi$ denote its kernel. Then there exists an induced short
    exact sequence of groups
    \begin{displaymath}
        0 \longrightarrow I \stackrel{\iota}{\longrightarrow} \GL \widetilde R
        \stackrel{\widetilde \phi}{\longrightarrow} \GL R \longrightarrow 0
    \end{displaymath}
    where
    \begin{displaymath}
    \begin{split}
        \iota (r) = 1+r;\\
        \widetilde \phi (r) = \phi (r).
    \end{split}
    \end{displaymath}
\end{prop}
\begin{proof}
    First we need to check that everything is well defined. Since by hypothesis
    we have that $\phi$ is a square-zero extension, then
    \begin{displaymath}
        \iota(r)\iota(-r) = (1+r)(1-r) = 1
    \end{displaymath}
    hence $\iota(-r) = \iota(r)^{-1}$; again, by hypothesis
    \begin{displaymath}
        \iota(r) \iota(r') = (1+r)(1+r') = 1+r+r' = \iota(r+r'),
    \end{displaymath}
    thus $\iota$ is a group homomorphism.
    Since $\phi$ is a ring homomorphism, it preserves units, therefore
    $\widetilde \phi$ is well defined.

    The sequence is also exact. Clearly, $\iota$ is injective. To see that
    $\phi$ reflects units, let $a, \ b \in \GL R$ be such that $ab=1$, and let
    $\alpha , \ \beta \in \widetilde R$ be such that $\phi(\alpha)=a$ and
    $\phi ( \beta ) = b$. It follows from
    \begin{displaymath}
        \phi ( \alpha \beta - 1 ) = ab - 1 = 0
    \end{displaymath}
    that $\alpha \beta -1 \in I$; hence, as $(\alpha \beta -1)^2 = 0$ we
    have $$\alpha \beta ( 2 - \alpha \beta) = 1$$ proving that both $\alpha$
    and $\beta$ are units in $\widetilde R$.
    Now, since 
    \begin{displaymath}
        \ker \widetilde \phi = \widetilde  \phi ^{-1} ( \{1\} ) =
        \phi ^{-1} ( \{1\} ) = 1+I = \iota (I)
    \end{displaymath}
    the sequence is exact.
\end{proof}

\section{Fundamentals of higher commutative algebra}\label{prelim}

\subsection{Symmetric monoidal \cats}
\begin{notat}
  We denote by $\einf\operad$ the \cat $\nerve(\mathrm{Fin}_*)$, that is,
  the nerve of the category of finite pointed sets.
  Given any $n \in \bbN$, and any $1 \leq i \leq n$, we denote by $\rho^i \colon
  \langle n \rangle \to \langle 1 \rangle$ the function sending all elements
  of $\langle n \rangle$ to (the basepoint) $0$,
  with the exception of the element $i$.
\end{notat}

\begin{defi}
  A \emph{symmetric monoidal \cat} is the datum of an \cat $\scrC\operad$
  together with a coCartesian fibration of simplicial sets $p\colon \scrC\operad
  \to \einf\operad$ satisfying the following ``Segal condition''.
  \begin{enumerate}
  \item[$\bullet$] For every $n\geq 0$, and every $0\leq i\leq n$ the functors
    $$\rho^i_! \colon \scrC\operad_{\langle n \rangle} \to \scrC\operad_{\langle 1
    \rangle}$$
  induced by the functions $\rho^i$ and the coCartesian fibration $p$,
  determine an equivalence
    $$\left( \rho^i_! \right)_{i=1}^n \colon
    \scrC\operad_{\langle n\rangle}
    \stackrel{\sim}{\longrightarrow}\prod_{i=1}^n
    \scrC\operad_{\langle 1 \rangle}.$$
  \end{enumerate}
  We denote by $\scrC$ the \cat
  $\scrC_{\langle 1 \rangle}\operad$, and, slightly abusing terminology,
  we say that $p$ exhibits a \emph{symmetric monoidal structure} on $\scrC$,
  and that $\scrC$ itself is a \emph{symmetric monoidal \cat}.
\end{defi}

We refer the reader to \cite[Chapter 2]{HA} for a discussion about how this
definition gives a homotopy coherent generalization of
the ordinary notion of a symmetric monoidal category. In particular, if $\scrC
\operad \to \einf\operad$ is a symmetric monoidal structure on $\scrC$, it
follows from the definitions that there exists a uniquely (up to canonical
isomorphism) determined bifunctor, denoted
$\otimes\colon \scrC \times \scrC \to \scrC$,
encoded by the symmetric monoidal structure (see \cite[2.0.0.6, 2.1.2.20]{HA}).

\begin{defi}
  Let $\scrC$ be a symmetric monoidal \cat. We say that $\scrC$ is
  \emph{closed} if, for each $C \in \scrC$, the functor
  $$\scrC \simeq \scrC \times \Delta^0 \xto{\id \times C}
    \scrC \times \scrC \xto{\otimes} \scrC$$
  (informally given by $D \mapsto D \tensor C$) admits a right adjoint.
\end{defi}

\begin{defi}
  Let $\scrC$ and $\scrD$ be symmetric monoidal \cats with symmetric
  monoidal structures $p\colon \scrC\operad \to \einf\operad$ and $q\colon
  \scrD\operad \to \einf\operad$.
  \begin{enumerate}
    \item A \emph{lax symmetric monoidal functor} is given by a map
    of $\infty$-operads $F\colon \scrC\operad \to \scrD\operad$ (i.e. a
    morphism of
    \cats over $\einf\operad$, carrying $p$-coCarte\-sian lifts of inert
    morphisms of $\einf\operad$ to $q$-coCarte\-sian morphisms in
    $\scrD\operad$).

    \item A \emph{symmetric monoidal functor} is given by a morphism of \cats
    over $\einf\operad$ carrying $p$-coCartesian morphisms to $q$-coCartesian
    morphisms.
  \end{enumerate}
\end{defi}

We sometimes abuse notation, and refer to (lax) symmetric monoidal
functors indicating only the underlying functor between the underlying
$\infty$-cat\-e\-gories.

\begin{ex}\label{symmetric-monoidal-cats}
  We have the following two notable examples (see also Definition \ref{spectra}):
\begin{enumerate}
  \item\label{cartesian-structure}
  \cite[Section 2.4.1]{HA} Given any \cat $\scrC$ with finite products,
  it has a symmetric monoidal structure, denoted
  $$\scrC^{\times} \to \einf\operad,$$
  encoding its Cartesian product.
  \item \cite[4.8.2.14]{HA} The \cat $\spaces _*$ of pointed spaces has a
  symmetric monoidal structure, denoted
  $$\spaces _* ^{\wedge}\to \einf\operad ,$$
  encoding the smash product of pointed spaces.
\end{enumerate}
\end{ex}

\begin{defi}
  Let $\scrC$ be a symmetric monoidal \cat. We let $\calg(\scrC)$ denote the
  full subcategory of sections $\einf\operad \to \scrC\operad$ of the structure
  map $\scrC\operad \to \einf\operad$ spanned by lax symmetric monoidal
  functors. We refer to $\calg(\scrC)$ as the \cat of \emph{commutative algebra
  objects} of $\scrC$.
\end{defi}

\begin{defi}
  Let $\scrC$ be an \cat with finite products. A \emph{commutative monoid
  object} of $\scrC$ is given by a functor $M\colon\einf\operad \to \scrC$ such
  that the morphisms $M(\langle n \rangle) \to M(\langle 1 \rangle)$ induced
  by the inert morphisms $\langle n \rangle \to \langle 1 \rangle$ exhibit
  $M(\langle n \rangle)$ as an $n$-fold product of $M(\langle 1 \rangle)$.
  We let $\mon(\scrC)$ denote the full subcategory of $\Fun(\einf\operad,
  \scrC)$ spanned by the commutative monoids of $\scrC$.
\end{defi}

As it is clear, in an \cat with finite products, it is possible to define
both commutative algebra objects with respect to the Cartesian symmetric
monoidal structure and commutative monoid objects; in fact, the two
definitions agree.

\begin{prop} \cite[2.4.2.5]{HA}
  Let $\scrC$ be an \cat with finite products, considered as a symmetric
  monoidal \cat with the Cartesian structure of Example
  \ref{symmetric-monoidal-cats}.\ref{cartesian-structure}.
  Then, there is an equivalence of \cats
  $$\calg(\scrC) \simeq \mon(\scrC).$$
\end{prop}

\begin{rem}\label{forget-algebras}
  A commutative algebra object $A\colon\einf\operad \to \scrC\operad$ of a
  symmetric
  monoidal \cat $\scrC$ determines an object $A(\langle 1 \rangle)$ of $\scrC$
  together with a homotopy coherent analogue of the structure of a commutative
  algebra object of an ordinary symmetric monoidal category. In fact, the
  assignment $A \mapsto A(\langle 1 \rangle)$ extends to a forgetful functor
  $$\calg(\scrC) \to \scrC.$$
  With a little abuse of notation, we often denote $A(\langle 1 \rangle)$
  just by $A$,
  and refer to it as a commutative algebra object (or a commutative monoid
  object, if the symmetric monoidal structure on $\scrC$ is Cartesian).
\end{rem}

\begin{rem}\label{ordinary-monoids}
  It follows from the definitions, that, if $M$ is a commutative monoid in
  an \cat $\scrC$, then the underlying object $M$ in the homotopy category
  $\ho\scrC$ is a commutative monoid object (in the ordinary sense).
\end{rem}

\begin{prop} \label{comm-groups} \cite[1.1]{ggn15}
  Let $\scrC$ be an \cat with finite products, and let $M$ be a commutative
  monoid object of $\scrC$. Then the following conditions are equivalent:
  \begin{enumerate}
  \item $M$ admits an inversion map for the multiplication induced by the
  commutative algebra object structure.
  \item The commutative monoid object of $\ho\scrC$ underlying $M$ is a
  group object.
  \end{enumerate}
\end{prop}

\begin{defi}
  Let $\scrC$ be an \cat with finite products; we say that a commutative
  monoid object $M \in \mon(\scrC)$ is a \emph{commutative group object} if
  it satisfies the equivalent conditions of Proposition \ref{comm-groups}. We
  write $\grp(\scrC)$ to denote the full subcategory of $\mon(\scrC)$
  consisting of commutative group objects.
\end{defi}

We refer the reader to \cite{ggn15} for other equivalent characterizations of
commutative group objects in an \cat, and for a detailed treatment of some of
its properties.

\begin{rem}\label{prototypical-ex}
  As we observed in the introduction, the
  \cat of spaces plays in the homotopy coherent world the same role that the
  ordinary category of sets plays in the ordinary case. Accordingly,
  we denote the \cats
  $\mon(\spaces)$ and $\grp(\spaces)$ just
  by $\mon$ and $\grp$, respectively. We sometimes refer to $\mon$ as the
  \cat of \emph{$\einf$-spaces} or as the \cat of
  \emph{$\einf$-monoids}, and we refer to $\grp$ as the \cat of
  \emph{$\einf$-groups}.

  Similarly, as the \cat of spectra is the homotopy coherent analogue of the
  category of Abelian groups, we denote the \cat $\calg(\Sp)$ just
  by $\calg$, and refer to it as the \cat of \emph{$\einf$-rings}.
\end{rem}

The difference between the \cat of $\einf$-groups and the \cat of spectra is
something that has no counterpart in the ordinary case.
As recalled below (see Remark \ref{infinite-loops}), the \cat
of $\einf$-groups is equivalent to the full subcategory $\Sp^{cn} \subset \Sp$
of connective spectra, which is not stable, and the smallest stable
presentable \cat containing it is precisely $\Sp$. In a sense, looking for an
homotopy coherent analogue of the Abelian category of Abelian groups,
one can start with
the more ``naive'' \cat of $\einf$-groups (the straightforward generalization
of the category of Abelian groups), and then the price to pay to have
stability (the generalization of the notion of being Abelian we want to
work with) is to add
nonconnective spectra to the picture, which (as far as the analogy with the
ordinary case goes) can be thought of as a sort of technical nuisance.

\begin{prop}\label{symm-mon-adj}
  Let $p\colon\scrC\operad \to \einf\operad$ and $q\colon\scrD\operad \to
  \einf\operad$
  be symmetric monoidal \cats, and let $F\colon\scrC\operad \to \scrD\operad$
  be a symmetric monoidal functor, such that the underlying functor
  $F_{\langle 1 \rangle}\colon\scrC\to\scrD$ admits a right adjoint. Then $F$
  admits a right adjoint $G\colon\scrD\operad \to \scrC\operad$, which is lax
  symmetric monoidal.
\end{prop}
\begin{proof}
  This follows immediately from \cite[7.3.2.1]{HA}.
\end{proof}

Adjunctions of the kind of Proposition \ref{symm-mon-adj} constitute a
homotopy coherent analogue of ordinary symmetric monoidal adjunctions, hence
we will adopt the same terminology in this context.

\begin{defi}
  Let $p\colon\scrC\operad \to \einf\operad$ and $q\colon\scrD\operad \to \einf\operad$
  be symmetric monoidal \cats. We say that an adjunction
  $$F : \scrC\operad \rightleftarrows \scrD\operad : G$$
  is a \emph{symmetric monoidal adjunction} if the left adjoint is symmetric
  monoidal and the right adjoint is lax symmetric monoidal.
\end{defi}

We now pose our attention to the case of presentable \cats.

\begin{prop}\cite[4.1, 4.4]{ggn15}
  Given a presentable \cat $\scrC$, the \cats $\mon(\scrC)$ and
  $\grp(\scrC)$ are presentable; moreover, there are functors
  $$\scrC \to \mon(\scrC) \to \grp(\scrC)$$
  which are left adjoint to the respective forgetful functors.
\end{prop}

Recall that, given a presentable \cat $\scrC$, we can define its
stabilization $\mathrm{Exc}_*(\spaces_*^{\mathrm{fin}}, \scrC)$ (which is again
presentable), denoted $\Sp(\scrC)$, related to $\scrC$ by an adjunction
$$\susinftyp : \scrC \rightleftarrows \Sp(\scrC): \ominfty$$
where the left adjoint is called the \emph{suspension spectrum} functor (see
\cite[Section 1.4.2]{HA} for details).

\begin{prop}\cite[4.10]{ggn15}\label{ladjs-factorization}
  Let $\scrC$ be a presentable \cat. The suspension spectrum functor
  $\susinftyp\colon\scrC \to \Sp(\scrC)$ factors as a composition of left
  adjoints
  $$\scrC \xto{(-)_+} \scrC _* \xto{\free} \mon(\scrC) \xto{(-)\gp} \grp(\scrC)
  \xto{\Binf} \Sp(\scrC)$$
  each of which is uniquely determined by the fact that it commutes with the
  corresponding free functor from $\scrC$.
\end{prop}

When $\scrC$ is a presentable symmetric monoidal \cat, the above chain of
adjunctions can be enhanced to a chain of symmetric monoidal adjunctions.

\begin{prop}\label{smash-product}\cite[5.1]{ggn15}
  Let $\scrC$ be a presentable closed symmetric monoidal \cat. The \cats
  $\scrC _*$,
  $\mon(\scrC)$, $\grp(\scrC)$ and $\Sp(\scrC)$ all admit closed symmetric
  monoidal structures, which are uniquely determined by the requirement that
  the respective free functors from $\scrC$ are symmetric monoidal. Moreover,
  each of the functors
  $$\scrC _* \to \mon(\scrC) \to \grp(\scrC) \to \Sp(\scrC)$$
  uniquely extends to a symmetric monoidal left adjoint.
\end{prop}

\begin{defi}\label{spectra}
  The stabilization
  $\Sp(\spaces)$ of the \cat of spaces is
  the \emph{\cat of spectra} $\Sp$. The product
  encoded by the symmetric monoidal structure $$\Sp\operad \to \einf\operad$$
  induced on $\Sp$ by Proposition \ref{smash-product} is commonly referred
  to as the \emph{smash product of spectra}.
\end{defi}

\begin{notat}\label{notat-adj-appendix}
  In most cases, we omit notations for the forgetful functors. In the
  special case $\scrC = \spaces$, we adopt the following notations
  and terminology:
\begin{enumerate}
  \item We denote the right adjoint to $\Binf$ as $\ominfty_a$ and refer
  to it as the \emph{underlying additive $\einf$-space} functor.
  \item We denote the composite left adjoint $\Binf \circ (-)\gp \circ
    \free$
  by $$\susinfty\colon\spaces_* \to \Sp$$
  and denote its right adjoint, with the usual abuse of notation, by
  $\ominfty$.
  \item We denote the composite left adjoint $\free \circ (-)_+$ by
  $$\free_+ \colon \spaces \to \mon.$$
\end{enumerate}
\end{notat}

\begin{rem}\label{infinite-loops}\cite[5.2.6.26]{HA}
  The \cat $\grp$ of $\einf$-groups and the \cat $\Sp^{cn}$ of connective
  spectra are equivalent. The equivalence is given by the functor $\Binf$
  defined above, whose essential
  image is exactly $\Sp^{cn}$. In particular,
  the composite functor $$\Sp \xto{\ominfty_a} \grp \xto{\sim} \Sp^{cn}$$
  is equivalent to the truncation functor $\tau_{\geq 0}$, right adjoint
  to the inclusion $\Sp^{cn} \to \Sp$.
\end{rem}

\begin{rem}\label{additive-spectra}
  As $\Sp$ is a stable \cat, \cite[1.1.3.5]{HA} implies that it is also an
  additive \cat (see \cite[Section 2]{ggn15} for the definition and a detailed
  discussion of this property). Therefore, by \cite[2.8]{ggn15}, every
  spectrum admits a commutative group structure, and all maps between spectra
  respect this structure; to be more precise, the forgetful map
  $\grp(\Sp) \to \Sp$ is an equivalence.
  In particular, given any spectrum $X$, there exists a map
  $\inv : X \to X$ such that the composition
  $$X \xto{\Delta} X \oplus X \xto{\id \oplus \inv} X \oplus X\xto{\mu} X$$
  (where $\Delta$ is the diagonal map, and $\mu$ is the
  multiplication induced by the commutative algebra structure on $X$) is
  nullhomotopic.
\end{rem}

The above remark is consistent with the fact that we think of $\Sp$ as the
homotopy coherent analogue of the Abelian category of Abelian groups, and
can be interpreted as saying that all spectra admit a ``homotopy
coherent commutative group structure''.

\subsection{Modules over commutative algebra objects}
Let $\scrC$ be a symmetric monoidal \cat, and let $R \in \calg(\scrC)$ be a
commutative algebra object in $\scrC$. Analogously to what happens in the
ordinary case, it is possible to define an \cat $\Mod_R(\scrC)$ of $R$-modules
which, under mild assumptions, is again symmetric monoidal. In
\cite[3.3, 4.5]{HA} a few equivalent models for this
\cat are constructed, together with a detailed discussion of their
equivalence. For future reference, we recollect in this section some results
about these \cats, whose formulation (and validity) are independent of the
specific model, referring the reader to \cite{HA} for details and proofs.

\begin{notat}
  Let $\scrC$ be a symmetric monoidal \cat, and let $R \in \calg(\scrC)$ be a
  commutative algebra object in $\scrC$. We denote by $\Mod_R(\scrC)$
  the \emph{\cat of $R$-modules}. If $\scrC = \Sp$,
  we denote $\Mod_R(\scrC)$ just by $\Mod_R$.
\end{notat}

\begin{prop}\label{prop-symm-mon-modules}\cite[4.5.2.1]{HA}
  Let $\scrC$ be a symmetric monoidal \cat. Assume that $\scrC$ admits
  geometric realizations of simplicial objects, and that the tensor product
  $\tensor\colon \scrC \times \scrC \to \scrC$ preserves geometric realizations
  of simplicial objects separately in each variable.
  Let $R \in \calg(\scrC)$ be a
  commutative algebra object in $\scrC$. The \cat $\Mod_R(\scrC)$ admits a
  symmetric monoidal structure
  $$\Mod_R(\scrC)^{\tensor_R} \to \einf\operad$$
  whose symmetric monoidal product is also called the \emph{relative tensor
  product}.
\end{prop}

\begin{prop}\label{extension-restriction-scalars}
  Let $\scrC$ be a symmetric monoidal \cat as in Proposition
  \ref{prop-symm-mon-modules}.
  Let $A$ and $B$ be commutative algebra objects in $\scrC$, and let
  $f \colon A \to B$ be a morphism in $\calg(\scrC)$. Then,
  there exists a symmetric monoidal adjunction
  $$- \tensor_A B : \Mod_A(\scrC) \rightleftarrows \Mod_B(\scrC) : f^!$$
  with the left adjoint given by the relative tensor product $M \mapsto M
  \tensor_A B$. We refer to such left adjoint as the \emph{extension
  of scalars} functor, and to the right adjoint as the
  \emph{restriction of scalars} functor.
\end{prop}
\begin{proof}
  The existence of the adjunction on the underlying \cats follows from
  \cite[4.6.2.17]{HA}, which also implies that the left adjoint is given
  by the relative tensor product. By \cite[4.5.3.2]{HA}, the left adjoint is
  symmetric monoidal. Finally, it follows from \cite[7.3.2.7]{HA} that the
  right adjoint is lax monoidal.
\end{proof}

We sometimes omit notation for the restriction of scalars functor.

\begin{rem}
  It follows from \cite[4.8.2.19]{HA} that the smash product
  $$\tensor\colon\Sp\times\Sp\to\Sp$$ preserves small colimits separately in each
  variable; in particular, $\Sp$ satisfies the hypotheses of Proposition
  \ref{prop-symm-mon-modules}.
\end{rem}

In ordinary commutative algebra, the category of modules over a ring is
Abelian. Analogously, the \cat of modules over an $\einf$-ring is stable.

\begin{prop}\label{stability-of-modules}\cite[7.1.1.5]{HA}
  Let $R$ be an $\einf$-ring. Then, the \cat $\Mod_R$ of $R$-modules is
  stable.
\end{prop}

\nocite{francisPHD}

\bibliography{standard}
\bibliographystyle{amsalpha}

\end{document}